\documentclass[11pt]{article}
\usepackage{amssymb,amsfonts,amsmath,latexsym,epsf,tikz,url}
\usepackage[usenames,dvipsnames]{pstricks}
\usepackage{pstricks-add}
\usepackage{epsfig}
\usepackage{pst-grad} 
\usepackage{pst-plot} 
\usepackage[space]{grffile} 
\usepackage{etoolbox} 
\makeatletter 
\patchcmd\Gread@eps{\@inputcheck#1 }{\@inputcheck"#1"\relax}{}{}
\makeatother

\newtheorem{theorem}{Theorem}[section]
\newtheorem{proposition}[theorem]{Proposition}
\newtheorem{observation}[theorem]{Observation}

\newtheorem{corollary}[theorem]{Corollary}

\newcommand{\qed}{\hfill $\square$\medskip}

\begin{document}

\title{Independent domination stability in graphs}

\author{ \small 	
		S. Alikhani$^{1}$, M. Mehraban$^{1}$,
A. Zakharov$^{2,3}$,
 H.R. Golmohammadi$^{2,3}$ 
}

\date{\today}

\maketitle

\begin{center}
	
$^1$Department of Mathematical Sciences, Yazd University, 89195-741, Yazd, Iran

$^{2}$Novosibirsk State University, Pirogova str. 2, Novosibirsk, 630090, Russia 

	$^3$Sobolev Institute of Mathematics, Ak. Koptyug av. 4, Novosibirsk,
	630090, Russia

	\bigskip
	{\tt alikhani@yazd.ac.ir  ~~Mazharmehraban2020@gmail.com ~~ a.zakharov3@g.nsu.ru ~~h.golmohammadi@g.nsu.ru}

\end{center}

\begin{abstract}
  A non-empty set $S\subseteq V (G)$ of the simple graph $G=(V(G),E(G))$ is an independent dominating
 set of $G$ if every vertex not in $S$ is adjacent with some vertex in $S$ and the vertices of $S$ are pairwise non-adjacent.  The independent domination number of $G$, denoted by $\gamma_i(G)$, is the minimum size of all independent dominating sets of $G$. The independent domination stability, or simply $id$-stability of $G$ is the minimum number of vertices whose removal changes the independent domination number of $G$.
 In this paper, we investigate properties of independent domination stability in graphs. In particular, we obtain several bounds and obtain  the independent domination stability of some operations of two graphs. 
\end{abstract}


\medskip
\noindent{\bf AMS Subj.\ Class.}:  05C05, 05C69.
\section{Introduction}

The various different domination concepts are
well-studied now, however new concepts are introduced frequently and the interest is growing
rapidly. We recommend two fundamental books \cite{8,9}. 
 Let $G=(V(G),E(G))$ be a simple of finite orders graph, i.e., graphs are undirected with no loops or
parallel edges and with finite number of vertices.   
For any vertex $v\in V(G)$, the  open neighborhood of $v$ is the
set $N(v)=\{u \in V (G) | uv\in E(G)\}$ and the  closed neighborhood of $v$
is the set $N[v]=N(v)\cup \{v\}$. For a set $S\subseteq V(G)$, the open
neighborhood of $S$ is $N(S)=\bigcup_{v\in S} N(v)$ and the closed neighborhood of $S$
is $N[S]=N(S)\cup S$. The private neighborhood $pn(v,S)$ of $v\in S$ is defined by 
$pn(v,S)=N(v)-N(S-\{v\})$. Equivalently, $pn(v,S)=\{u\in V| N(u)\cap S=\{v\}\}$. Each vertex in
$pn(v,S)$ is called a private neighbor of $v$. The degree of a vertex $v$ is denoted by $\deg(v)$ and is equal to $|N(v)|$.   A leaf of a tree is a vertex of degree 1, while a support vertex is a vertex adjacent to a leaf. A double star is a tree that contains exactly two vertices that are not
end-vertices; necessarily, these two vertices are adjacent. 
  A non-empty set $S\subseteq V(G)$ is a  dominating set if $N[S]=V$, or equivalently,
every vertex in $V(G)\backslash S$ is adjacent to at least one vertex in $S$ and  the minimum cardinality of all dominating sets of $G$ is called  the  domination number of $G$ and is denoted by $\gamma(G)$. A domination-critical vertex in a graph $G$ is a vertex whose removal decreases the domination number. It is easy to observe that for any graph $G$ we have $\gamma(G)-1\leq \gamma(G+e)\leq \gamma(G)$
for every edge $e\not\in E(G)$. Sumner and Blitch in \cite{14} have defined domination critical
graphs. A graph $G$ is said to be domination critical, or $\gamma$-critical, if
$\gamma(G+e) =\gamma(G)-1$ for every edge $e$ in the complement $G^c$ of $G$.
A graph is said to be domination stable, or $\gamma$-stable, if $\gamma(G) =\gamma(G+e)$ for
every edge $e$ in the complement $G^c$ of $G$. For detailed information and results regarding the concept of domination critical graphs, we refer interested readers to the papers \cite{5,14,15}. Bauer et al introduced the concept of domination stability in graphs in 1983 \cite{3}. After then, was studied by Rad, Sharifi and  Krzywkowski in \cite{11}. Stability for different types of domination parameters has been investigated in the literature, for example, in \cite{2,6,10,12,wcds}. This subject has considered and studied for another graph parameters. For example see \cite{saeid, saeid2}.   

Recall that independent set in a graph $G$ is a set of pairwise non-adjacent vertices. A maximum independent set in $G$ is a largest independent set and its size is called independence number of $G$ and is denoted by $\alpha(G)$. An independent dominating set of $G$ is a set that is both dominating and independent in $G$. Equivalently, an independent dominating set is a maximal independent set. The independent domination number of $G$, denoted by $\gamma_i(G)$, is the minimum size of all independent dominating sets of $G$.  An independent dominating set of cardinality $\gamma_i(G)$ is called a $\gamma_i$-set. 
 A graph $G$ is independent domination critical, or $i$-critical if $\gamma_i(G-v)\neq \gamma_i(G)$ for every $v\in V(G)$. The independent domination and independent domination critical graphs have been reported in \cite{4,7,13,16}.
This paper is organized as follows. 

In the next section, we introduce the  independent
domination stability ($id$-stability) and  present some of its properties.  In Section 3, we obtain  some bounds for the independent domination stability of graphs. In Section 4, we study the independent domination stability of some operations of two graphs.

\section{ Introduction to $id$-stability}

In this section, we first introduce independent version of domination
stability. The independent
domination stability, or $id$-stability of a graph $G$ is the minimum size of a
set of vertices whose removal changes the independent domination number. To establish a standard notation, we denote the independent domination stability by $st_{id}(G)$. The $i^-$-stability of $G$, denoted by $st_{id}^-(G)$, is the minimum number of vertices whose removal decreases the independent domination number. Likewise, the $i^+$-stability of $G$, denoted by $st_{id}^+(G)$, is the minimum number of vertices whose removal increases the independent domination
number. Thus the independent domination stability of a graph $G$ is 
$st_{id}(G)= \min \{st_{id}^-(G), st_{id}^+(G)\}$. Note that considering the null graph $K_0$, which has no vertices, as a graph, the independent domination stability of a non-trivial graph is always defined. For instance, $st_{id}(G)=|V(G)|$, occurs if and only if $G=K_n$.

Now we  determine domination stability for some classes of graphs. To aid our discussion, we first state some straightforward observations as follows.

\begin{observation}
	\begin{enumerate}
		\item [(i)] 
If $G$ is  a star or a double star, then $st_{id}(G)=1$.
\item[(ii)] 
If $K_{m,n}$ is  a bipartite complete graph where $m\ge n\ge 2$, then $st_{id}(K_{m,n})=1$.

	\end{enumerate}
\end{observation}

The friendship (or Dutch-Windmill) graph $F_n$ is a graph that can be constructed by the coalescence of $n$
copies of the cycle graph $C_3$ of length $3$ with a common vertex. The Friendship Theorem of Paul Erd\"{o}s,
Alfred R\'{e}nyi and Vera T. S\'{o}s \cite{erdos}, states that graphs with the property that every two vertices have
exactly one neighbour in common are exactly the friendship graphs. 
Let $n$ and $q\geq 3$ be any positive integer and  $F_{q,n}$ be the {\em generalized friendship graph}  formed by a collection of $n$ cycles (all of order $q$), meeting at a common vertex.  (see Figure \ref{figtetragons}). 

\begin{figure}[ht]
	\hspace{2.0cm}
	\includegraphics[width=10.5cm,height=2.5cm]{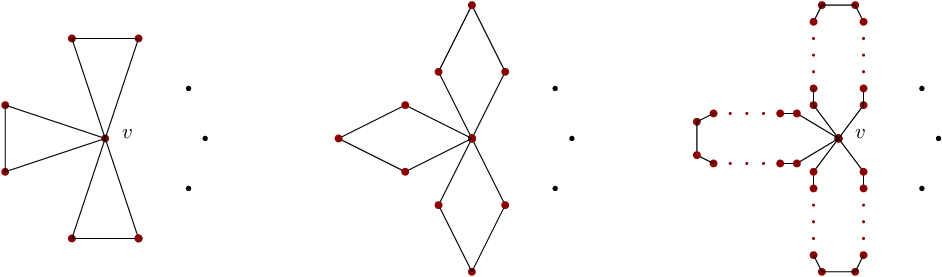}
	\caption{\label{figtetragons} The flowers  $F_{n},~ F_{4,n}$ and $F_{q,n}$, respectively. }
\end{figure}

The $n$-book graph $(n\geq2)$ (Figure \ref{book}) is defined as the Cartesian product $K_{1,n}\square P_2$. We call every $C_4$ in the book graph $B_n$, a page of $B_n$. All pages in $B_n$ have a common side $v_1v_2$.   We shall compute the distinguishing stability number  of $B_n$. 
The following observation gives the independent domination number of  the friendship graph and the book graph. 

\begin{figure}
	\begin{center}
		\hspace{.7cm}
		\includegraphics[width=0.5\textwidth]{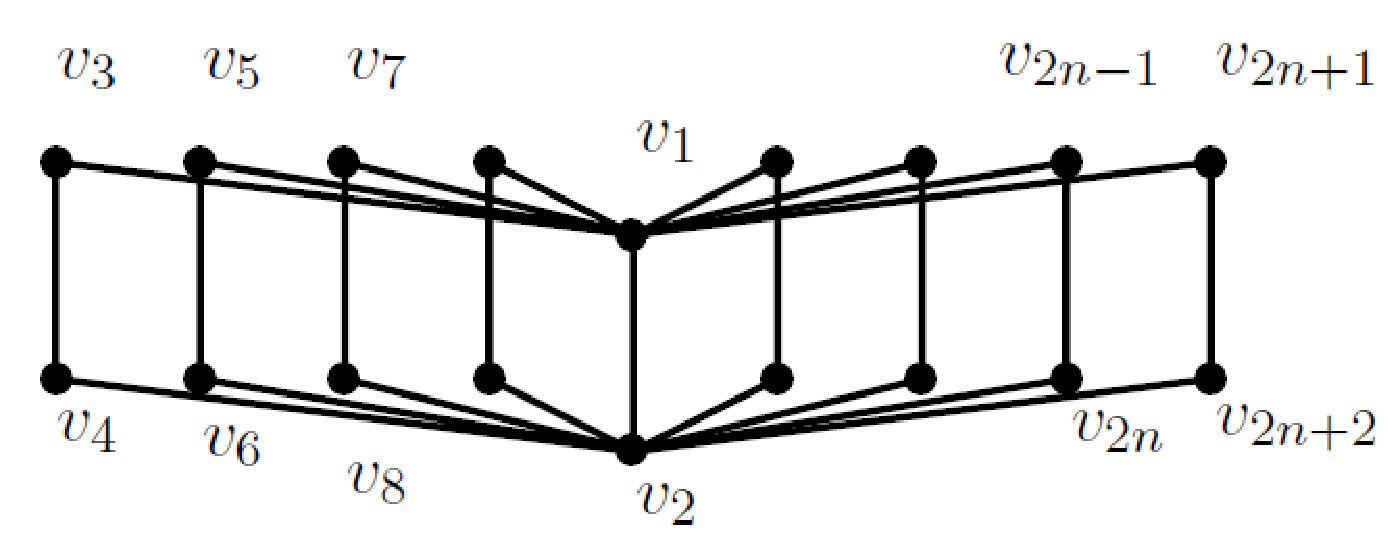}
		\caption{\label{book}  Book graph $B_n$.}
	\end{center}
\end{figure}  

\begin{observation}
	\begin{enumerate} 
\item[(i)] 
If $F_n$ is a friendship graph, then $\gamma_i(F_n)=1$.

\item[(ii)] $\gamma_i(F_{4,n})=\gamma_i(F_{5,n})=\gamma_i(F_{6,n})=n+1$

	\item[(iii)] 
	If $B_n$ is a book graph, then $\gamma_i(B_n)=n$.
			\end{enumerate} 
			\end{observation}

\begin{proposition}
		\begin{enumerate} 
			\item[(i)] 
	If $F_n$ is a friendship graph, then $st_{id}(F_n)=1$.
	
	\item[(ii)] For every $k\geq 3$, $st_{id}(F_{k,n})=1$.
	
	\item[(iii)] 
$st_{id}(B_n)=2$. 
	
	\end{enumerate} 
\end{proposition}
\begin{proof}
	\begin{enumerate} 
		\item[(i)] By removing the center vertex of $F_n$, we have $nK_2$ and obviously $\gamma_i(F_n)\neq\gamma_i(nK_2)=n$. So we have the result. 
		
		\item[(ii)] 
		If $v$ is the center vertex of $F_{k,n}$, then 
		$\gamma_i(F_{k,n}-v)=\gamma_i(P_{k-1})=
		\lfloor\frac{k+1}{3}\rfloor\neq n+1$, so we have the result. 
		
		\item[(iii)] By removing two non-central vertices $u,v$ in the page of $B_n$ we have $\gamma_i(B_{n}-\{u,v\})=n-1$. \qed
	\end{enumerate}
	\end{proof} 
		
\begin{theorem}
There exist graphs $G$ and $H$ with the same independent domination number such that $|st_{id}(G)-st_{id}(H)|$ is very large.
\end{theorem} 
\begin{proof}
	Let $G=K_n$  with $\gamma_i(G)=1$ and $st_{id}(G)=n$. Let $H$ denote the graph $K_{1,n-1}$ with $\gamma_i(H)=1$ and $st_{id}(H)=1$. We have $|st_{id}(G)-st_{id}(H)|=n-1$ and so the result follows. \qed
	\end{proof}

Before investigating the independent domination stability of paths ans cycles, we state  the following observation.

\begin{observation}{\rm \cite{7}}\label{pathcy}
	 For the path and cycle, $\gamma_i(P_n)=\gamma_i(C_n)=\lfloor\frac{n+2}{3}\rfloor$.
 \end{observation}

The following result establishes an upper bound on the independent domination stability in terms of $\delta(G)$.

\begin{observation}\label{delta} 
	 For any graph $G$,  $st_{id}(G) \leq  \delta(G)+1$.
\end{observation}

We first determine the $i$-stability of paths.

\begin{proposition}
    If $P_{n}$ is a path of order $n$, then
 $st_{id}(P_n)=\left\{
 \begin{array}{cc}
 
 2   &\quad n\equiv 2 ~(\mbox{mod } 3)\\
 1    &\quad otherwise\\
 \end{array}\right.
. $

\end{proposition}
\begin{proof} We consider  the following cases:  

\noindent{\bf Case 1.} Suppose first that $n\equiv 0 ~(\mbox{mod } 3)$. We observe that $\gamma_i(P_{n}-v)=\gamma_i(P_n)+1$, where $v$ is a support vertex. So, $st_{id}(P_n)=1$. 

\noindent{\bf Case 2.} Assume that $n\equiv 1 ~(\mbox{mod } 3)$. If $v$ is a leaf, then $\gamma_i(P_{n}-v)=\gamma_i(P_n)-1$, and it follows that $st_{id}(P_n)=1$.  

\noindent{\bf Case 3.} Suppose  that $n=3k + 2$ for some integer $k$. By Observation \ref{pathcy}, we have $\gamma_(P_n)=k + 1$. We 
show that the removal of $v$ does not change the independent domination number. For this purpose, we consider the following subcases.

\noindent{\bf Subcase 3.1.} If $v$ is a leaf, then, $\gamma_i(P_{n}-v)=\gamma_i(P_{n-1})=k+1=\gamma_i(P_n)$.

\noindent{\bf Subcase 3.2.} Let $deg(v)=2$ and  $P_n-v$
 consists of two connected components $P_{n_1}$ and $P_{n_2}$ such that $n_1 + n_2=n-1$. Without loss of generality, we first 
assume that $n_1\equiv 0 ~(\mbox{mod } 3)$ and $n_2\equiv 1 ~(\mbox{mod } 3)$, and so 
$$\gamma_i(P_{n}-v)=\gamma_i(P_{n_1})+\gamma_i(P_{n_2})= \lfloor\frac{n_1+2}{3}\rfloor+\lfloor\frac{n_2+2}{3}\rfloor=\frac{n_1}{3}+\frac{n_2+2}{3}=k+1=\gamma_i(P_n).$$ Analogously, we can obtain that $\gamma_i(P_{n}-v)=\gamma_i(P_n)$ when $n_1\equiv 2 ~(\mbox{mod } 3)$ and $n_2\equiv 2 ~(\mbox{mod } 3)$.

Now, by Subcases 3.1 and 3.2, we conclude that $st_i(P_n) \ge 2$. Moreover, recall that by Proposition \ref{delta}, we have  $st_{id}(P_n) \leq 2$. Consequently, $st_{id}(P_n)=2$ and the result follows. \qed

\end{proof}

Next, we determine the $i$-stability of cycles.

\begin{proposition}

If $C_{n}$ is  a cycle of order $n$, then
  $st_{id}(C_n)=\left\{
 \begin{array}{cc}
 3  &\quad n\equiv 0 ~(\mbox{mod } 3)\\
 2   &\quad n\equiv 2 ~(\mbox{mod } 3)\\
 1    &\quad n\equiv 1 ~(\mbox{mod } 3)\\
 \end{array}\right.
 .$
 
\end{proposition}
\begin{proof} We consider the following cases:  

\noindent{\bf Case 1.} We  assume that $n=3k+1$ for some integer $k$. Then for any vertex $v$ we have $\gamma_i(C_{n}-v)=\gamma_i(P_{n-1})=k=\gamma_i(C_n)-1$, and so $st_{id}(C_n)=1$. 

\noindent{\bf Case 2.} Suppose that $n=3k+2$ for some integer $k$. For any vertex $v$ we have $\gamma_i(C_{n}-v)=\gamma_i(P_{n-1})=k+1=\gamma_i(C_n)$, and it holds that $st_i(C_n)\ge 2$. Now we consider that $u$ and $v $ are two adjacent vertices. Then, $\gamma_i(C_{n}-u-v)=\gamma_i(P_{n-2})=k=\gamma_i(C_n)-1$, which implies $st_{id}(C_n)=2.$

\noindent{\bf Case 3.} Assume that $n=3k$ for some integer $k$. It is easy to see that the removal of any vertex does not change the independent domination
number. Since $C_n-v=P_{n-1}$, by Proposition \ref{pathcy} we have $st_i(P_{n-1})=2$, and so  $st_i(C_n)=3$.\qed
\end{proof}

\section{Bounds on the independent domination stability }
In this section, we establish some bounds for the independent domination stability of a graph.

 Here, we study the relationship between the $id$-stabilities of graphs $G$ and $G-v$, where $v\in V(G)$. Also we obtain upper bounds for $st_{id}(G)$. 
 \begin{proposition}\label{minus}
 	Let $G$ be a graph and $v$ be a vertex of $G$, then $$st_{id}(G)\leq   st_{id}(G-v)+1.$$
 \end{proposition}
 \begin{proof} 
 	  If $\gamma_i(G)=\gamma_i(G-v)$, then $st_{id}(G)\leq   st_{id}(G-v)+1$. If $\gamma_i(G)\neq \gamma_i(G-v)$, then $st_{id}(G)= 1$, and so we have the result.\qed
 	  \end{proof} 
 
 Using Proposition \ref{minus} and mathematical induction, we have 
 \[
 st_{id}(G)\leq   st_{id}(G-v_1-\cdots - v_s)+s,
 \]
  where $1\leq   s \leq   n-2$ and $n=\vert V(G)\vert$. Only using this formula we can get different upper bounds for $st_{id}(G)$ of graph $G$. We state some these upper bounds in the following theorem which for the proof  of each case, it is sufficient to remove the vertices of $G$ until the induced subgraph which  stated in the hypothesis, appears. Next using $st_{id}(G)\leq   st_{id}(G-v_1-\cdots - v_s)+s$ and the value of $st_{id}(G-v_1-\cdots - v_s)$, we can have the result. 
 
 \begin{theorem}\label{multipartitestabil}
 	Let $G\neq K_n$ be a simple graph of order $n\geq2$.
 	\begin{itemize}
 		\item[(i)]  $st_{id}(G)\leq   n-1$.
 		\item[(ii)] If graph $G$ has  the star graph $K_{1,t}$ as the induced subgraph with $t \geq3$, then  $st_{id}(G)\leq   n-t$.
 	\item[(iii)]   		For any graph $G$, $st_{id}(G) \leq n-\Delta(G)$. 
 	\end{itemize}
 \end{theorem}

\begin{theorem}\label{stab&dist}
	If  $G$ is a graph of order $n$, then $st_{id}(G)\leq  n+1-2\gamma_{i}(G)$.
\end{theorem}
\begin{proof} 
Set $st_{id}(G)=k$. Thus for every $i$ vertices of $G$  ($1\leq   i \leq   k-1$), say $v_1,\ldots, v_{i}$,  we have $\gamma_{id}(G)=\gamma_{id}(G-v_1)=\cdots = \gamma_{id}(G-v_1-\cdots - v_{i})$.  It is known that   the independent domination  number of a graph is at most equal to  half of its order, so $\gamma_i(G)=\gamma_i(G-v_1-\cdots -v_{k-1})\leq   \frac{n-k+1}{2}$. \qed
\end{proof} 

\begin{corollary}
	Let $G$ be a graph of order $n\geq 2$. If $st_{id}(G)=n-1$, then $\gamma_i(G)= 1$.
\end{corollary}
\begin{proof} 
	 It is a direct consequence of Theorem  \ref{stab&dist}. \qed
	 \end{proof}

\begin{theorem}\label{r}
	For any graph $G$ with $\gamma_i(G)\ge 2$ we have $st_{id}(G)\leq \frac{n}{\gamma_i(G)}$.
\end{theorem}

\begin{proof}
	Let $I$ be an independent dominating set of $G$. We consider the following cases.
	
	\noindent{\bf Case 1.} We  assume that $u$ is  a vertex of $I$ which has no  private neighbor in $V(G)-I$. Then the removal
	of $u$ in $V(G)-I$ decreases the independent domination number, which holds that $st_i(G)\leq \frac{n}{\gamma_i(G)}$.

	\noindent{\bf Case 2.} Suppose that $w$ is   a vertex of $I$ with minimum number of private neighbors in $V(G)-I$. Then the removal
	of $w$ and its private neighbors in $V(G)-I$ decreases the independent domination number, which follows that $st_{id}(G)\leq \frac{n}{\gamma_i(G)}$.\qed
	
\end{proof}

\begin{proposition}{\rm\cite{7}}\label{g}
If $G$ is an isolate-free graph of order $n$, then 
$$\gamma_i(G)\leq n+2-\gamma(G)-\lceil\frac{n}{\gamma(G)}\rceil.$$
 \end{proposition}

 \begin{proposition}{\rm\cite{8}}\label{l}
For every connected graph $G\neq K_1$ we have $\gamma(G)\leq \frac{n}{2}$.
 \end{proposition}

\begin{proposition}\label{h}
If  $k \in \mathbb{N}$ and $\gamma_i(G)\ge k > 1$, then there is no isolate-free graph $G$ of order $n$ with $\gamma_i(G)\ge 2$ and $st_{id}(G)=n-k$.
\end{proposition}

\begin{proof}
	Assume that $G$ is an isolate-free graph with $\gamma_i(G)\ge 2$ and $st_{id}(G)=n-k$. By Proposition \ref{g}, we have $n > \frac{nk}{\gamma_i(G)} \ge \frac{k(\gamma_i(G)-2+\gamma(G)+\frac{n} {\gamma(G)})}{\gamma_i(G)}$. Using Proposition \ref{l}, it can be seen that $n-k=st_{id}(G) > \frac{k\frac{n}{2}}{\gamma_i(G)}$. Since $k\ge 2$, then $n-k=st_{id}(G) > \frac{n}{\gamma_i(G)}$, contradicting Theorem \ref{r}.\qed
\end{proof}

\begin{proposition}\label{q}
 For any graph $G$ with $\gamma_i(G)\ge 2$, we have 
 \[
 st_{id}(G) \leq  \min\{\delta(G)+1, n-\delta(G)-1\}.
 \]
\end{proposition} 

\begin{proof} 
	Let $I$ be an independent dominating set. Suppose that  there is a vertex $u \in I$ such that $epn(u,I)=\emptyset$, so  $st_{id}(G)=1$. Now, we  assume that for each vertex $u\in I$ we have $epn(u,I)\neq\emptyset$. Suppose that $v \in I$, and let $X$ be the
set of private neighbors of $v$ in $V(G)-I$. Therefore $\gamma_i([G(X \cup \{v\}])=1<\gamma_i(G)$. By Observation \ref{delta}, we have the result. \qed
\end{proof}

Here, we shall state Nordhaus–Gaddum type inequalities for the sum of the independent domination stability of a graph and its
complement. Note that if $\gamma_i(G)=1$, then $st_{id}(\overline{G})=1$, and thus we obtain the following result.

\begin{observation}
If $G$ is a graph of order $n$ with $\gamma_i(G)=1$ or $\gamma_i(\overline{G})=1$, then $st_{id}(G)+st_i(\overline{G})\leq n+1$, and this bound is sharp for the complete graphs.
\end{observation}

\begin{theorem}
    If $G$ is a graph with $\gamma_i(G)\ge 2$ and $\gamma_i(\overline{G})\ge 2$, then
  $$st_{id}(G)+st_{id}(\overline{G})\leq\left\{
 \begin{array}{cc}
n     &\quad   n=2k\\
n-1   &\quad   n=2k+1

 \end{array}\right.
 $$
 \end{theorem} 
   \begin{proof} 
   	If  $n=2k+1$ for some integer $k$, then by Theorem \ref{r} we have $st_{id}(G)+st_{id}(\overline{G})\leq \frac{n-1}{2}+\frac{n-1}{2} =n-1$. Now, we suppose that $n=2k$.
Using Proposition \ref{q}, we observe that 
\begin{eqnarray*}
st_{id}(G)+st_{id}(\overline{G})&\leq& min\{\delta(G) + 1, n-\delta(G)-1\}+ min\{\delta(\overline{G}) + 1, n-\delta(\overline{G})-1\} \\
&&\leq \frac{n}{2}+\frac{n}{2}=n.
\end{eqnarray*}\qed
\end{proof}

\section{$id$-stability of some operations of two graphs}
In this section, we study the independent domination stability of some operations of two graphs. First we consider the join of two graphs. 
The join $ G_1+ G_2$ of two graphs $G_1$ and $G_2$ with disjoint vertex sets $V_1$ and $V_2$ and edge sets $E_1$ and $E_2$ is the graph union $G_1\cup G_2$ together with all the edges joining $V_1$ and $V_2$. 

\begin{observation}\label{join}
	If  $G_1$ and $G_2$ are  nonempty graphs, then $$\gamma_i(G_1+G_2)=\min\{\gamma_i(G_1),\gamma_i(G_2)\}.$$
	\end{observation}
	\begin{proof}
		By the definition, every  $\gamma_i$-set  $D_1$ of  $G_1$ (or $\gamma_i$-set $D_2$ of  $G_2$), is a $\gamma_i$-set of $G_1+G_2$. So we have result.\qed
	\end{proof} 
	
	By Observation \ref{join}, we have the following result.
	
	\begin{theorem}
		If  $G_1$ and $G_2$ are  nonempty graphs, then $$st_{id}(G_1+G_2)=\min\{st_{id}(G_1),st_{id}(G_2)\}.$$
					\end{theorem}

Here, we recall the definition of lexicographic product of two graphs.  
 For two graphs $G$ and $H$, let $G[H]$ be the graph with vertex
 set $V(G)\times V(H)$ and such that vertex $(a,x)$ is adjacent to vertex $(b,y)$ if and only if
 $a$ is adjacent to $b$ (in $G$) or $a=b$ and $x$ is adjacent to $y$ (in $H$). The graph $G[H]$ is the
 lexicographic product (or composition) of $G$ and $H$, and can be thought of as the graph arising from $G$ and $H$ by substituting a copy of $H$ for every vertex of $G$ (\cite{DAM}).
 
 The following theorem gives the independent domination number of $G[H]$. 
 
 \begin{theorem}\label{lexico}
 If  $G$ and $H$ are two  graphs, then $$\gamma_i(G[H]) =\gamma_i(G)\gamma_i(H).$$
 \end{theorem}
\begin{proof}
An independent dominating set in $G[H]$ of minimum cardinality,  arises by choosing an independent dominating
set in $G$ of cardinality $\gamma_i(G)$,  and then, within each copy of $H$ in $G[H]$, 
choosing an independent dominating set in $H$ with cardinality $\gamma_i(H)$. Thus, we have the result. \qed	
	\end{proof} 

By Theorem \ref{lexico}, we have the following result.

\begin{corollary}
 If  $G$ and $H$ are two  graphs, then $st_{id}(G[H])=\min\{st_{id}(G),st_{id}(H)\}$. 
\end{corollary}

Now, we obtain the $id$-stability of corona of two graphs. We first state and prove the following theorem.
\begin{theorem}\label{corona}
	If  $G$ and $H$ are two  graphs, then $\gamma_i(G\circ H) =|V(G)|\gamma_i(H)$.
\end{theorem}
\begin{proof}
	An independent dominating set in $G\circ H$ of minimum cardinality,  arises by choosing an independent dominating
	set with minimum cardinality in  each copy of $H$ in $G\circ H$.  So we have the result. \qed	
\end{proof} 

By Theorem \ref{corona}, we have the following result.

\begin{corollary}
	If  $G$ and $H$ are two  graphs, then $st_{id}(G\circ H)=1$. 
	\end{corollary}
\begin{proof}
	By Theorem \ref{corona}, $\gamma_i(G\circ H) =|V(G)|\gamma_i(H)$, so by removing one vertex from $G$, the order of $G$ and so   $\gamma_i(G\circ H)$ will be changed. Therefore, we have the result. 
	\end{proof} \\

\noindent{\bf Acknowledgement.} 
The work of Hamidreza Golmohammadi and Alexey Zakharov is supported by the
Mathematical Center in Akademgorodok, under agreement No. 075-15-2022-281 with
the Ministry of Science and High Education of the Russian Federation.


\begin{thebibliography}{99}



				
	\bibitem{saeid} S. Alikhani and M.R. Piri, 	On the edge chromatic vertex stability number of graphs,  AKCE Int. J. Graphs Comb., 20(1) (2023) 29-34.	
	
	\bibitem{saeid2} S. Alikhani and S. Soltani, Stabilizing the distinguishing number of a graph, Commun. Algebra 46 (12) (2018)  5460-5468.	
				
               
                \bibitem{2} G. Asemian, N. Jafari Rad, A. Tehranian, H. Rasouli:
On the total Roman domination stability in graphs. AKCE Int. J. Graphs Comb. 18(3): 166-172 (2021)
				

                \bibitem{3} D. Bauer, F. Harary, J. Nieminen and C. L. Suffel, Domination alteration sets in graphs, Discrete Math. 1983; 47(2-3): 153– 161.

            


				\bibitem{4} M. Edward, A. Finbow, G.MacGillivray,  S. Nasserasr, Independent domination bicritical graphs. Australas. J. Comb. 2018, 72, 446–471.
				
				\bibitem{erdos} P.  Erdős, A.  R\'{e}nyi,\& V. T.  S\'{o}s,    On a problem of graph theory, { Studia Sci. Math. Hungar.},  1 (1966) 215--235.
				

                \bibitem{5} J. Fulman, D. Hanson and G. MacGillivray, Vertex domination-critical graphs, Networks 25 (2) (1995), 41–43


                \bibitem{6} A. Gorzkowska, M. A. Henning, M. Pilśniak, E. Tumidajewicz, Paired domination stability in graphs, Ars Math. Contemp. 22 (2022) \#P2.04; doi:10.26493/1855-3974.2522.eb3

							
				\bibitem{7}  W. Goddard, M.A. Henning, Independent domination in graphs: A survey and recent
				results. Discrete Math., 313 (7) (2013), 839-854.
				
				
				\bibitem{8}  T.W. Haynes, S.T. Hedetniemi, P.J. Slater, Fundamentals of domination in graphs, Marcel Dekker, NewYork  (1998).

                \bibitem{9} T.W. Haynes, S.T. Hedetniemi, P.J. Slater Domination in Graphs: Advanced
                 Topics. Marcel Dekker, New York (1998).

              \bibitem{10} M. A. Henning and M. Krzywkowski, Total domination stability in graphs, Discrete Appl.
              Math. 236 (2018), 246–255, doi:10.1016/j.dam.2017.07.022.

				\bibitem{11} N. Jafari Rad, E. Sharifi and M. Krzywkowski, Domination stability in graphs, Discrete Math. 339 (2016), 1909-1914.
				
				\bibitem{DAM} S. Jahari, S. Alikhani, On the independent domination polynomial of a graph, Discrete Appl. Math. 289 (2021) 416-426. 

               \bibitem{12}  Z. Li, Z. Shao and S.-j. Xu, 2-rainbow domination stability of graphs, J. Comb. Optim. 38
               (2019), 836–845, doi:10.1007/s10878-019-00414-0

               \bibitem{13} K. Kuenzel and D. F. Rall, On independent domination in direct products, Graphs Combin. (2023) 39:7
https://doi.org/10.1007/s00373-022-02600-0



\bibitem{wcds} M. Mehryar and S. Alikhani, Weakly connected domination stability in graphs, Adv. Appl. Math. Sci., 16 (2) (2016) 79-87.

               \bibitem{14}  D.P. Sumner, P. Blitch, Domination critical graphs, J. Combin. Theory Ser. B 34 (1983), 65–-76.

               \bibitem{15} D.P. Sumner, Critical concepts in domination, Discrete Math. 86 (1990) 33–46.

               \bibitem{16} P. Wu, H. Jiang, S. Nazari-Moghaddam, S.M. Sheikholeslami, Z. Shao, L. Volkmann, Independent domination stable trees and unicyclic graphs, Mathematics 7 (2019), no. 820, 17 pp.
				
			
			

    
				
			\end{thebibliography}
\end{document}